\documentclass[12pt]{amsart}
\usepackage{stix} 
\usepackage{tikz}
\usetikzlibrary{positioning,arrows}
\usepackage{cool} 
\usepackage{microtype} 
\usepackage{cite} 
\usepackage{mathtools}
\usepackage{amsthm,thmtools}
\usepackage{hyperref} 
\usepackage[capitalise]{cleveref}
\mathtoolsset{showonlyrefs,showmanualtags}

\declaretheorem[style = plain]{theorem}   
\declaretheorem[style = plain,      sibling = theorem]{corollary,lemma,proposition}
\declaretheorem[style = definition, sibling = theorem]{definition, example}
\declaretheorem[style = remark]{remark,problem}

\DeclareMathOperator{\dist}{dist}
\DeclareMathOperator{\Ext}{Ext}
\DeclareMathOperator{\Hom}{Hom}

\newcommand{\Bh}{{\ensuremath{B(\Hilbert)}} }
\newcommand{\compact}{\ensuremath{\mathcal K}}
\newcommand{\compacts}{\compact}

\newcommand{\corona}{\ensuremath{\mathcal Q}} 

\newcommand{\Cstar}[1][]{\textsl{C*}{#1}} 
\newcommand{\dd}{\star\star}

\newcommand{\fsum}[2][{\FiniteSums}]{\Sum{#2}{#1}}
\newcommand{\Hilbert}{\ensuremath{\mathcal{H}}}
\newcommand{\Id}{\text{Id}}
\newcommand{\KKnuc}{\KK_{{\text{nuc}}}}
\newcommand{\KKw}{\KK_{{w}}}
\newcommand{\KK}{{\textsl{KK}}} 
\newcommand{\KL}{{\textsl{KL}}}
\newcommand{\Kh}{\compacts(\Hilbert)}
\newcommand{\Mbb}{\ensuremath{\frac{\Mb}{B}} }
\newcommand{\Mbkbk}{\ensuremath{{\Mbk}/({B\tensor\compact})}}
\newcommand{\Mbk}{\ensuremath{\Mult (B\tensor\compact)}}
\newcommand{\Mb}{\ensuremath{\Mult (B)} }
\newcommand{\Mult}{\mathcal{M}}
\newcommand{\Oinf}{\ensuremath{\O{\infty}} }
\newcommand{\Wstar}[1][]{\textsl{W*}{#1}} 
\newcommand{\Z}{\mathbb{Z}} 
\newcommand{\geqcomp}{\ensuremath{\underset{\widetilde{\quad}}{\succ}}}
\newcommand{\integers}{\mathbb Z}
\newcommand{\isom}{\cong}

\newcommand{\predual}{\phantom{}_{*}}

\newcommand{\tensor}{\otimes}
\newcommand{\vi}{\ensuremath{v_i}}
\newcommand{\vj}{\ensuremath{v_j}}
\renewcommand{\L}[1]{{\ensuremath{{ L}_{#1}} }} 
\renewcommand{\O}[1]{{\ensuremath{{\mathcal O}_{#1}} }} 
\renewcommand{\inf}{\infty} 
\renewcommand{\j}{\ensuremath{\iota}}
\renewcommand{\star}{\phantom{}^{*}}

\begin{document}

\title{\sc Separating points by inclusions of \Oinf}
\author{Dan Z. Ku\v{c}erovsk\'{y}}
\email{dkucerov@unb.ca}
\address{Dept. of Math.\\University of New Brunswick at Fredericton\\Fredericton, NB\\Canada  E3B 5A3}
\begin{abstract} We define a basis property that an inclusion of \Cstar-algebras $\Oinf\subset A$ may have, and give various conditions for the property to hold. Some applications are considered. We also give a characterization of open projections in a corona algebra.
\end{abstract}\maketitle

\section{Introductory Comments}\label{sect:Introduction}

It occurs frequently in \Cstar-algebraic problems that one has, or wishes to construct, a unitally embedded copy of \Oinf
in a given \Cstar-algebra $A.$
Very often, such embeddings are expressed
in terms of a set of preferred generators for \Oinf and the images of the generators in $A$ become the focus of attention.
To clarify the property of
the images that we will investigate here, we begin with the following
definitions.
\begin{definition} Let $A$ be a unital \Cstar-algebra. An (infinite) Cuntz
family of isometries in $A$ is a sequence $\{\vi\}_{i=1}^{\inf}$
of isometries in $A$
with the property that
$$\sum_{i=1}^{n} \vi\vi\star \leq \mathbf{1}_{A}, \qquad 1\leq n<\inf.$$
\end{definition}
If the \Cstar-algebra $A$ is \Bh for some infinite dimensional Hilbert
space \Hilbert, then an infinite Cuntz family of isometries generates \Oinf
as defined in \cite{Cuntz1977}.  Furthermore by Theorem 1.12 [ibid.], any two infinite
Cuntz families of isometries in \Bh generate isomorphic copies of
\Oinf. As Cuntz points out immediately after Theorem
1.12 [ibid.], this fact proves that \Oinf is algebraically simple. As a consequence of these
observations, every infinite Cuntz family in a \Cstar-algebra $A$ may
be viewed as the image of an infinite Cuntz family of isometries in
\Bh under an embedding of \Oinf .
\begin{definition}  Let $A$ be a \Cstar-algebra containing an infinite Cuntz
family of isometries $\{\vi \}_{i=1}^{\inf}$
and let $p_i := \vi\vi\star.$ Then we say that
$\{\vi \}_{i=1}^{\inf}$
\emph{separates the points} of $A$ when the only $x\in A$ such that
$p_i xp_j = 0$ for all $i$ and $j$ is $x = 0.$ \end{definition}
With these definitions in hand, we may formulate the fundamental problem that we investigate here as:
\begin{problem}[the basis problem]\label{problem:sep.points}
    Given an infinite Cuntz family of isometries in a \Cstar-algebra, determine when it separates the points of the algebra.\end{problem}
After presenting some technical lemmas in Section  \ref{sect:Some useful lemmas}, we present
three vignettes in which we describe how to solve \textit{\cref{problem:sep.points}}
in special situations. The first appears in Section \ref{sect:Bh}, where we
study Cuntz families in $\Bh.$ The second, appearing in Section
\ref{sect:Oinf}, investigates embeddings of \Oinf in \O{n} for finite $n.$ The third vignette, which is about embeddings of
\Oinf in a corona algebra, is in section \ref{sect:the corona}, where we also give a characterization of open projections in a corona algebra (see \cref{th:open.projs.for.corona.algs}). The final section, Section \ref{sect:extensions}, is devoted to
applications.
 Since, for the most part, distinct techniques need to be used for the different cases,  each section contains a short introduction of the techniques used in that section.

The separation of points problem is closely related to the classic problem of trying to write a given element as an infinite matrix of operators.
Suppose that we break up a given element $x$ into orthogonal summands using elements $\vi,$ much as above. Thus, consider
\begin{equation*}
  \begin{pmatrix}
    v_1 x v_1\star  & v_1 x v_2\star & \cdots & \\
    v_2 x v_1\star &  v_2 x v_2\star& \cdots & \\
  \vdots  & \vdots  & \ddots &   \\
  \end{pmatrix}
  \end{equation*}
where the matrix may be infinite. The two questions that arise have to do with the convergence and
with the injectivity of the construction. 
We will focus on the question of injectivity. In other words, if we have found that $v_ix v_j\star  = 0$ for all $i$ and $j$, can we conclude that $x=0$?
Results on injectivity are useful for deciding when an element of a \Cstar-algebra can be decomposed as a infinite matrix (the decomposition problem). The converse problem, briefly discussed in \cref{sec:matrix.decomp}, of deciding when  an infinite matrix of general form does converge to an element of a specific \Cstar-algebra (the convergence problem) is quite hard. See for example Kadison \cite[6.6]{KR2}, or the classic Schur's theorem \cite[p.174]{maddox} for some results in this direction.

The above problem depends on exactly how \Oinf is embedded in $A,$ and has some depth even if the target algebra is just $\Bh.$  We  point out some useful and tractable special cases. Moreover, we  construct interesting examples of embeddings of $\Oinf$ that separate points.
In \Cstar-algebra theory, stability is the most obvious route for unitally embedded copies of \Oinf to arise, however, it is not the case that stability is required, see \cite{ElliottHandelman1989}.
The basis problem is related to the little-studied algebraical problem of inner annihilators, and what work there is in the algebraic literature \cite{lam2014} focuses on the case of inner annihilators of elements in von Neumann regular rings.  Because the algebraic problem seems quite difficult, we  will mostly use methods based on representations, states, and/or weak limits, rather than algebra. However, the trade-off is that in most cases we can only find sufficient conditions for a solution.

Some of Kadison's techniques \cite{kaftal1990,kadison1984} for partitioning the identity in masas, which he developed for an application to diagonalization of operator matrices, get used in section \ref{sect:Some useful lemmas} and section \ref{sect:extensions}.

\section{Some useful lemmas}\label{sect:Some useful lemmas}

\begin{definition}We say a set $S$ of positive elements contained in a \Cstar-algebra $A$ is \emph{strictly positive} if for every state $\rho$ of  $A$ we have $\rho(S)\not=\{0\}.$ In other words, for every state $\rho\in S(A)$ there  exists some element of $S$ on which $\rho$ is nonzero.  \label{def:strictly.positive}\end{definition}

We note a very simple and probably fundamental formula for the hereditary subalgebra generated by a collection of projections.
\begin{lemma} The hereditary subalgebra $H$ of generated by a given  collection $\{p_\lambda\}$ of   projections in a \Cstar-algebra $B$ is equal to $\left(pB\dd p\right) \cap B  $ with $p:=\sup_\lambda p_\lambda.$ \label{lem:supremum.formula} \end{lemma}
\begin{proof}   To every hereditary subalgebra $H\subset B$ corresponds a open projection in the double dual, which can be defined as a projection $p$ in the double dual such that $\left(pB\dd p\right) \cap B = H.$ See \cite[3.11.10]{pedersen.book}.
Since in our case the hereditary subalgebra is the smallest hereditary subalgebra containing  all of the projections $\{p_\lambda\},$ we can view $p$ as being the smallest projection in $B\dd$ satisfying the following two properties:
\begin{enumerate}\item $p\geq p_\lambda$ for all $\lambda,$ and
\item $p$ is open.
\end{enumerate}
The first of these conditions can be rephrased in terms of a well-known quantity. Indeed, the supremum  $\sup_\lambda  p_\lambda$ is by definition the smallest projection of $B\dd$ satisfying this condition. On the other hand, a projection $p_\lambda$ in $B$ is evidently an open projection for the corner $p_\lambda B p_\lambda$ that it generates in $B.$ This means that  the supremum  $ \sup_\lambda  p_\lambda$ is a supremum of open projections, and Akemann, using work of Effros, has shown that a supremum of open projections is open \cite[Prop. II.5]{Akemann1969}. Thus, $p:=\sup_\lambda  p_\lambda$ satisfies both the first and the second of the above conditions. Moreover, since it is the smallest projection satisfying the first condition, it is clearly the smallest projection satisfying both the first and the second condition. This shows that the hereditary subalgebra generated by the collection $\{p_\lambda\}$ has open projection $p=\displaystyle \sup_\lambda  p_\lambda,$ and the correspondence between open projections and hereditary subalgebras gives $H=\left(pB\dd p\right) \cap B ,$ as was to be shown.
\end{proof}

From this  lemma we can deduce an interesting corollary. We will only use it in the countable case, but it holds in general if stated carefully.
\begin{corollary} Let  $\{p_\lambda\}$ denote a set of   projections in a given \Cstar-algebra $B.$
The following are equivalent:
\begin{enumerate}\item The set of projections is strictly positive in the sense of  Definition \ref{def:strictly.positive},
\item the hereditary subalgebra generated by the set $\{p_\lambda\}$ of  projections is all of $B,$ and
\item the supremum $\displaystyle\sup_{\mathcal F} p_\lambda$ is equal to $1_{B\dd}.$
    \end{enumerate}\label{cor:three.equivalences.for.strict.positivity}\end{corollary}
\begin{proof} The equivalence of (1) and (2) follows from the fact that every hereditary subalgebra is an intersection of hereditary kernels of states \cite[1.5.2 and 3.13.5]{pedersen.book}. That (2) is equivalent to (3) follows from Lemma \ref{lem:supremum.formula} together with the observation that the open projection associated with a hereditary subalgebra $H$ is equal to $1_{B\dd}$ exactly  when $H$ is all of $B.$\end{proof}
If we restrict to the case of a countable set of orthogonal projections, the supremum of this orthogonal set in the projection lattice of the double dual coincides with the ultraweakly convergent sum of the projections in the double dual, by \cite[2.105]{AlfsenSchultz}, and so we have the interesting corollary that:
\begin{corollary}[\bf Criterion for strict positivity] Let  $\{p_i\}$ denote a countable set of pairwise orthogonal  projections in a given \Cstar-algebra $B.$
 The set of projections is strictly positive in the sense of  Definition \ref{def:strictly.positive} if and only if  $\,\fsum{p_i}  = 1_{B\dd}$ ultraweakly.\label{cor:when.is.a.sum.strictly.positive}\end{corollary}
 One might wonder if  norm convergence could not be used in the above corollary, but infinite sums of projections will not converge in norm. This leaves as options either a strict topology or a weak topology. Certainly, a natural setting for taking infinite sums of projections is the enveloping algebra, \textit{i.e.} the bi-dual.
 \begin{example} Consider $\L{\inf} (S^1,\mu)$ where $\mu$ is Haar measure. It is shown in  \cite[3.1.(2)]{DecomposabilityofvonNeumannalgebras} that there exist projections summing to 1 in the projection lattice of this \Wstar-algebra, and therefore the hereditary subalgebra generated by these projections, being the same as the norm-closed ideal generated, is all of $\L{\inf}(S^1,\mu)$. But evidently these projections do not sum, either in norm or strictly, to 1 there. We can give an explicit description of one such family of projections: consider the indicator functions $p_n$ of the closed intervals $[2^{-n} \pi,  2^{-n+1}\pi]_{n=0}^\inf.$ These functions define projections in the algebra and their pairwise products are functions supported on sets of measure zero. So these are orthogonal projections, and to check ultraweak convergence of the partial sums  to 1 means checking convergence under states given by integrating against Haar measure with respect to functions in $\L{1} (S^1,\mu).$
 \end{example}
Indeed, the above example generalizes to \Wstar-algebras with separable predual; in other words, $\Wstar$-algebras that are representable on separable Hilbert spaces, see \cite[p.69]{pedersen.book} :
\begin{lemma} Let $B$ denote a \Wstar-algebra that is representable on a separable Hilbert space.
Then there exists a countable family of orthogonal projections in $B$ that sum to 1 ultraweakly.
\label{lem:countablefamily}
\end{lemma}
\begin{proof} Let $\{p_\lambda\}$ be a maximal family of orthogonal projections. It is not possible to represent uncountably many orthogonal projections on a separable Hilbert space \cite[p.5]{Dixmier1957}, so this family is necessarily countable. On the other hand, the increasing net of partial sums $\fsum p_\lambda$ converges ultraweakly \cite[2.105]{AlfsenSchultz} to some projection, $P,$ in  $B.$ If this projection were not equal to 1, then we could have added $1-P$ to the family of orthogonal projections, but this would contradict maximality.
\end{proof}
Ultraweak convergence, when available, can be brought to bear as follows:
\begin{lemma}Let $A$ be a unital C*-algebra. Given $\Oinf\subset A$, the basis problem can be solved if the diagonal projections $\vi\vi\star$ sum ultraweakly to $1_A.$\label{prop:strictlypositive} \end{lemma}
\begin{proof} We are given that the sum $\fsum \vi\vi\star,$ converges ultraweakly to $1,$ in $A\dd.$ We note that the elements $\vi\vi\star$ belong to $A,$ as does the limit $1.$ To show that we can separate points means that we are to show that given an element $x\in A$ such that $\vi\vi\star x \vj\vj\star=0,$ then $x=0.$ But this is clear, because multiplication is  \cite[lemma 2.66]{AlfsenSchultz} ultraweakly continuous in each variable (separately).
Thus, summing over finite subsets with respect to $i$ and  taking ultraweak limits,  $x\vj\vj\star=0$ in $A\dd.$ Then, summing  over $j$ and taking ultraweak limits, $x=0.$ Thus $x$ is zero, as was to be shown. \end{proof}
The following corollary connects strict positivity with the property of separating points, showing that if the set $\{p_i\}$ is strictly positive, then the set of projections $\{p_i\}\subset B$ separates the points of $B.$ Generally, strict positivity is a sufficient but not necessary condition.
\begin{proposition} The basis problem can be solved for $\Oinf\subset B$ if
the subset $\{\vi\vi\star\}$ is strictly positive in $B,$ where $\vi$ denotes the generators of \Oinf.
\label{prop:can.solve.if.sp} \end{proposition}
\begin{proof} Combine Corollary \ref{cor:when.is.a.sum.strictly.positive} and Lemma \ref{prop:strictlypositive}. \end{proof}

The separation of points problem often has good behaviour with respect to  quotients:
\begin{proposition}[Quotients]\label{prop:quotients} Suppose that $\Oinf\subset B$ is such that the set
$\{\vi\vi\star\}$ is strictly positive in the unital \Cstar-algebra \textbf{}$B,$ where $\vi$ denotes the generators of $\Oinf.$  Suppose that $I$ is a two-sided proper ideal of $B.$ If we replace $B$ by the quotient $B/I,$ then the image of \Oinf in the quotient will separate the points of $B/I.$ \end{proposition}
\begin{proof} Strict positivity in $B$ implies strict positivity in any nontrivial  quotient $B/I.$ This is because a state of $B/I$ can be composed with the canonical quotient map to give a state of $B$, and thus is nonzero on  the given set. But then by \cref{prop:strictlypositive,cor:three.equivalences.for.strict.positivity} the image of the set of projections $\{\vi\vi\star\}$ separates the points of the quotient $B/I.$
\end{proof}
 \smallskip
\subsection{Semicontinuity applied to matrix decomposition}\label{sec:matrix.decomp}
Next we will use some facts about semicontinuity in \Cstar-algebras.
   It is well known that for ordinary real functions, if a function and its negation are both (lower) semicontinuous, then the function is continuous. This fact generalizes nicely to  the noncommutative setting as follows.
\begin{lemma}[\!\!\protect{\cite[th. 3.12.9]{pedersen.book}}]
  Let $A$ be a  \Cstar-algebra. Then $\Mult(A)_{sa}=(U_{sa})^m \bigcap (U_{sa})_m$ where $(U_{sa})^m$ (respectively $(U_{sa})_m$) denotes the set of limits of increasing (resp. decreasing) nets of self-adjoint elements of $A.$
\label{th:semicontinuity}\end{lemma}
The significance of the above lemma is that certain sums that can be defined in the double dual in fact define elements of the \Cstar-algebra or its multipliers.
In the presence of slightly more than a solution to the basis problem, we obtain the following theorem:

\begin{theorem}[\bf Continuity] Let $c_i$ denote elements of the unit ball of a unital  \Cstar-algebra $\corona,$ and let $p_i$ denote  projections  summing sequentially and ultraweakly to 1. Then
the sum $T:= \sum_1^\infty   p_i c_i p_i$ defines an element of  $\corona.$ \label{prop:inf.sum.for.T}\label{th:inf.sum.for.T}
\end{theorem}
\begin{proof} Consider the sum, as above:
\begin{equation}\label{eq:sum.in.corona}
  T:= \sum_1^\infty   p_i  c_i p_i
\end{equation}
where the $c_i$ are arbitrary elements of the positive unit ball of a unital \Cstar-algebra $\corona,$ and the sum converges in the double dual $\corona\dd.$
We note that the sum is monotone increasing, so that $T$ is a (lower) semi-continuous element of the double dual.
Replacing $c_i$ by $1-c_i$ we have
\begin{equation}\label{eq:sum.in.corona.complements}
  1-T = \sum_1^\infty   p_i  (1-c_i) p_i
\end{equation}
which shows that $1-T$ is also given by a monotone increasing sum. But then both $T$ and $-T$ are lower semicontinuous. Applying Lemma \ref{th:semicontinuity}  shows that the element $T$ defined by the sum is actually in the algebra $\corona,$ as was to be shown. This proof assumed that the $c_i$ were positive, but we can drop the restriction to positive elements by using the fact that a general element of a \Cstar algebra can be written as a linear combination of four positive elements.
\end{proof}
The above theorem provides a conditional expectation from a \Cstar-algebra $\corona$ onto the diagonal subalgebra (of diagonal elements) in $\corona.$ 
Additional properties of these elements will be discussed later, in section \ref{sect:extensions}.
We remark that a similar proof implies that $N$--diagonal infinite matrices also define elements of the \Cstar-algebra:
\begin{corollary} Let $c_{ij}$ denote elements of the unit ball of a unital \Cstar-algebra $\corona.$ Suppose that $c_{ij}=0$ if $\,|i-j|>N,$ for some fixed integer $N,$ and  let $(p_i)$ denote a countable set of orthogonal projections summing sequentially to 1 in the double dual. Then
the sum $T:= \sum_{i,j=1}^\infty   p_i  c_{ij} p_j$
 defines an element of the \Cstar-algebra \corona.
\end{corollary}

\section{Separating points and the Cuntz-Krieger property for $\Oinf\subset\Bh$}\label{sect:Bh}

As we will now show, the problem of separating points  reduces to known quantities in the case of $\Oinf\subset \Bh.$

A copy of \Bh containing a unital copy of \Oinf is equivalent to a representation of the \Cstar-algebra \Oinf on the given Hilbert space \Hilbert. One of the fundamental properties of representations of \Oinf is that they may or may not be essential. We  define a representation of \Oinf to be {\it essential} if the projections $\vi\vi\star$ sum sequentially and strongly to 1 in the representation.
It seems that this  property was first  observed, but not named, by Cuntz and Krieger in \cite[Remark 2.15]{CuntzKrieger1980}, and the term essential might be due to Arveson\footnote{The term essential is also used for representations whose range does not  intersect the compact operators. This is not what we mean here.}. His monograph \cite{Arveson1989}  gives  more information on the property of being essential; see also the pioneering papers of J\o rgensen and co-workers, \textit{e.g.} \cite{Jorgensen}, where this property is shown to have important applications. In the case of \Bh we have the following characterization of the separation of points property:
\begin{theorem} The basis problem for $\Oinf\subset \Bh$ can be solved if and only if, viewing the given inclusion map as a representation, the map is essential. \label{th:case.of.Bh}\end{theorem}
\begin{proof}  Let $p_i$ denote the diagonal projections in $\Oinf.$ If the representation is not essential, then $\sum p_i$ does not converge strongly to $1_\Bh.$
But then it converges to some projection $P$ in $\Bh,$ other than $1,$ and evidently $P$ acts as $1$ on each of the projections $p_i.$  But then $p_i (1-P) p_j=0,$ while $1-P\not=0,$ so that the basis problem has a counter-example. On the other hand, if the representation is essential, then we can apply Proposition \ref{prop:strictlypositive} to solve the basis problem. 
\end{proof}
Since every \Cstar-algebra is a sub-\Cstar-algebra of some $\Bh,$ the above case of \Bh is in a sense the most general case of the separation of points problem, and as we have seen, cannot always be solved affirmatively. Given an embedding of $\Oinf$ into $\Bh$, the above result gives a criterion for being able to solve the separation of points problem. It shows that there do exist embeddings $\Oinf\subset\Bh$ for which this problem can be solved affirmatively.
We complete the picture by giving an interesting geometrical  construction of an essential embedding of \Oinf into $\Bh.$
\begin{example}\label{ex:Linf.construction.in.multipliers} Given a separable Hilbert space $\Hilbert,$ Kadison shows \cite[Ex.5.1.6]{KR1} that there is a diffuse $\sigma$-finite masa $M\subset \Bh$ and that the unit $1\in M$ decomposes as $1_M = \displaystyle\bigvee p_i$ for some countable set of pairwise orthogonal projections $\{p_i\}\subset M.$ \footnote{In fact, $M\isom\L{\inf} ([0,1],\mu),$ where $\mu$ is Lebesgue measure.}
But then $\displaystyle\sum p_i =1_\Bh,$ sequentially and strongly in $\Bh.$
On the other hand, the masa is diffuse, so that $M\cap\Kh=\{0\},$ and hence none of the projections $\{p_i\}$ are compact in $\Bh.$ Murray and von Neumann showed that such projections are Murray-von Neumann equivalent in $\Bh,$ and are also equivalent to $1_\Bh.$ Thus these projections generate a copy of $\Oinf,$ and we have $\Oinf\subset\Bh,$ unitally and essentially.\end{example}

Strong convergence of a sequence of projections $p_i$ in $\Bh$ implies strict convergence of the projections $1\tensor p_i$ in $\Mbk.$ This is because strict convergence in the multiplier algebra is equivalent to the strong convergence of adjointable Hilbert module operators, in the appropriate Hilbert module sense, on $B\tensor\compacts.$  Multiplication is (separately) strictly continuous \cite[2.3.3]{WO}, so the proof of Lemma \ref{prop:strictlypositive} gives the following corollary:
\begin{corollary}  Let $\pi$ be an essential representation of $\Oinf.$ Then the basis problem can be solved for $\Oinf$ embedded in $\Mbk$ by the map $1\tensor \pi\colon \Oinf\to\Mbk.$  \label{cor:embed.essential}\end{corollary}
We remark that if  we define the multiplier algebra in yet a different way, as the idealizer of $A:=B\tensor\compacts$ in its double dual, then strict convergence is stronger (finer) than relative ultraweak convergence $\sigma(\Mult(A),A\predual).$
\section{On $\Oinf\subset\O{n}$}\label{sect:Oinf}

The case of $\Oinf\subset\O{n}$ has a more algebraic flavour.
In this section, we first show the existence of one natural embedding of $\Oinf$ into $\O{n}$  such that the diagonal projections of $\Oinf$ separate the points of $\O{n},$ and then give a more general result that applies to all unital embeddings.
Let $s_1,\ldots,s_n$ and $t_1,t_2,\ldots$ denote
 generators of $\O{n}$ and $\Oinf$, respectively.
 Define the embedding $f_{n,\infty}$
of $\Oinf$ into $\O{n}$ by
\begin{equation*}
\label{eqn:embedding}
f_{n,\infty}(t_{(n-1)k+i}):=s_n^k\,s_i\quad(k\geq 0,\ i=1,\ldots,n-1)
\end{equation*}
where we define $s_n^0:=\mathbf{1}$.
The existence of this natural embedding $f_{n,\infty}$ has been observed before, and our description follows \cite[Definition 1.4(ii)]{Kawamura2016}.

Now we are in a position to bring our Theorem \ref{th:case.of.Bh} to bear. To do this, we shall choose an essential representation of \Oinf and show that we can extend it to a representation of $\O{n}.$ It will then follow from Theorem \ref{th:case.of.Bh} that the basis problem can be solved for elements of \Bh in $\Oinf\subset\O{n} \subset\Bh,$ in which case the basis problem is solved for elements of $\O{n}.$

It remains to show that we can in fact extend a representation of $\Oinf$ to a representation of $\O{n}.$ In the classic theory of \Cstar-algebras, there is a standard process \cite[prop. 2.10.2]{Dixmier1964} for the extension of representations to a larger algebra, but it requires enlarging the given Hilbert space. In our more algebraic situation given by inclusions of Cuntz algebras, one can use constructions with generators to be able to extend representations of $\Oinf$ to representations of $\O{n}$ without enlarging the given Hilbert space. We recall the following known property of representations of Cuntz algebras.
Given a representation $(\Hilbert,\pi)$ of $\Oinf$, and an embedding $\Oinf\subset\O{n},$
we can define a representation  $(\Hilbert,\pi')$ of $\O{n}$ by
\begin{equation*}
\label{eqn:kawamuras}
\pi'(s_i)=
\left\{
\begin{array}{ll}
\pi(t_i)\quad (i=1,\ldots,n-1),&\\
\\
{\pi(I)-\sum_{j=1}^{\infty}\pi(t_{j}t_j^*)
+\sum_{j=1}^{\infty}\pi(t_{j+n-1}\,t_j^*)}&(i=n).
\end{array}
\right.
\end{equation*}
It can be verified that $\pi'(s_1),\ldots,\pi'(s_n)$ satisfy the Cuntz relations for $\O{n}$.

By definition, if $\pi=0$, then $\pi'=0$.
We summarize the construction in a lemma, for whose proof see \cite[Lemma 3.7]{KawamuraArxiv}.
\begin{lemma}
\label{lem.extension} Let $\Oinf$ be embedded by $f_{n,\infty}$ in $\O{n}.$
For any representation $\pi$ of $\Oinf$ there exists a representation $\pi'$ of $\O{n}$ that extends the original representation $\pi$ in the sense that $\pi'\circ f_{n,\infty}=\pi$.
\end{lemma}
This shows that we can indeed extend a representation of the smaller algebra, $\Oinf,$ to a representation of the larger algebra, $\O{n},$ without changing the ambient Hilbert space. Applying this construction to an essential representation of $\Oinf,$ the extended representation then separates the points of the larger algebra $\O{n}$ by Theorem \ref{th:case.of.Bh}.
We thus have the main result of this section:
 \begin{corollary} The basis problem can be solved for $\Oinf \underset{f_{n,\infty}}{\subset} \O{n}.$ \end{corollary}

 The above corollary appears to depend on the use of a very specific embedding map denoted by $f_{n,\infty}.$ However, recall that purely infinite, simple, unital, and nuclear \Cstar-algebras have been classified by their $K$-theory groups, adjusted slightly to take into account the position of the unit. The classification is moreover sufficiently functorial to also classify maps such as the one above. Lettting $f_{n,\infty}$ and $f$ denote some two injective unital maps of  $\Oinf$ into $\O{n},$ and passing to $K$-theory, we have
 \begin{center}
 \begin{tikzpicture}[auto]
\node (Z)                                        {$\integers$};  
\node (Z_n_1) [above right = of Z]  {$\integers_n$};
\node (Z_n_2) [below right = of Z]  {$\integers_n$};

\path[->] (Z) edge      node             {$f_{n,\infty}$}   (Z_n_1) 
                     edge      node [swap] {$f$}                    (Z_n_2)
(Z_n_1) edge [densely dotted,semithick] node   {$\alpha$}           (Z_n_2);
\end{tikzpicture}\end{center}
where the dotted line denots the map $\alpha$ that makes this diagram commutative. The map $\alpha$ is necessarily some automorphism of $\integers_n,$ and the classification theory lets us lift it to some automorphism, also denoted $\alpha,$ of $\O{n}.$ The significance of this is that our special map $\Oinf \underset{f_{n,\infty}}{\subset} \O{n}$ can be made equivalent to any other unital inclusion map  $\Oinf \underset{f}{\subset} \O{n}$ by an automorphism of $\O{n}.$ If the basis problem can be solved for one, it can be solved for the other, so we find that our corollary above implies an apparently much more general case:
 \begin{corollary}[Inclusions of \O{n}] The basis problem can be solved for any unital inclusion $\Oinf {\subset} \O{n}$\label{cor:final.result.on.On.into.Oinf}\end{corollary}

\section{The corona $\corona:=\Mbb$}\label{sect:the corona}
Corona algebras provide a huge supply of nontrivial orthogonal elements, but on the other hand, their representation theory is complex, due to the fact that a corona algebra often cannot  be represented faithfully on a separable Hilbert space. Furthermore, hereditary subalgebras of the corona are well-understood only in the $\sigma$-unital case.  These various aspects make corona algebras an interesting test case for our theory.

The abundance of projections in a corona algebra leads to some subtle simplifications in the theory of open projections in this case.
The term \textit{open projection} usually refers to projections of $A\dd$ in the universal representation of $A$ having the property that $p A\dd p \cap A$ is a hereditary subalgebra of $A.$ The universal representation is sufficiently large so that its projections distinguish all hereditary subalgebras of $A,$ in other words, every hereditary subalgebra can be written (uniquely) in terms of an open projection as shown.  If we consider instead a faithful nondegenerate representation and say that $p\in\Bh$ is \textit{open} if $p=\sup a_\alpha$ for some increasing net $(a_\alpha )\subset A_+,$ then, in general, we can distinguish just the hereditary subalgebras that arise as two-sided \textit{annihilators} of subsets of $A.$ But, in the special case of $\sigma$-unital hereditary subalgebras of a corona algebra, we can do better. The significance of the next theorem is that it lets us study open projections in a representation of our choice rather than in the large and complicated universal representation. This could for example be helpful for studying the Cuntz semigroup of a corona algebra through its open projections.

\begin{theorem} Let $B$ be the corona algebra of a $\sigma$-unital \Cstar-algebra. Let $\pi\colon B\to\Bh$ be a nondegenerate faithful representation of $B.$ Then, $\sigma$-unital hereditary subalgebras $H$ of $B$ are of the form  $p \Bh p \cap B$ for some $p\in\Bh.$
 \label{th:open.projs.for.corona.algs}\end{theorem}
\begin{proof} Suppose that the given $\sigma$-unital hereditary subalgebra $H$ happens to be the annihilator of a subset of the larger algebra $B.$ Then the desired result $H=p \Bh p \cap B$ holds by \cite[3.3 and 3.4]{Bice.arxiv}. On the other hand, by Pedersen's double annihilator theorem for corona algebras \cite[Th.15]{Pedersen1986}, every  $\sigma$-unital hereditary subalgebra $H$ of a corona algebra can in fact be written as the annihilator of some subset of the corona algebra $B.$ So the result holds in general, as was to be shown.
\end{proof}

Leaving for now the topic of open projections, what about the basis problem for a corona algebra?

In fact, we can find one solution to the basis problem straightaway.
\begin{example}\label{ex:Linf.construction.in.corona}
Let  $B$ be a  stable and $\sigma$-unital \Cstar-algebra. Recall that in \cref{ex:Linf.construction.in.multipliers,cor:embed.essential} we constructed a strongly convergent sum of projections $p_i$ in the multiplier algebra $\Mb:$
\begin{center}\begin{tikzpicture}[auto]
  \node (Mb) at (0,0)  {$\{p_i\}\subset\L{\infty} ([0,1],\mu)\subset\Mb$};    
  \node (BH1) [right=of Mb] {$B(\Hilbert_1).$};   
\draw[right hook->] (Mb) -- (BH1);
  \end{tikzpicture}\end{center}
  Applying the canonical quotient map $\pi\colon\Mb\to\Mbb,$ the kernel of  $\pi$ does not nontrivially intersect the subalgebra $M$ given by the image of $\L{\infty} ([0,1],\mu)$ in $\Mb.$ (See \cref{ex:Linf.construction.in.multipliers}.)  But then this means that the vertical map $\pi$ in the commuting triangle
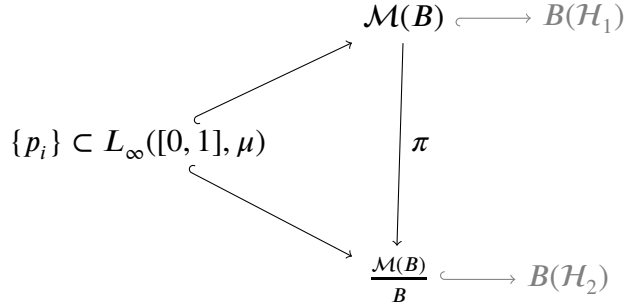
\begin{center}\begin{tikzpicture}[auto]
  \node (L) at (0,0)  {$\{p_i\}\subset\L{\infty} ([0,1],\mu)$};    
  \node (Mb) [above right=of L]  {$\Mb$};
  \node (Mbb) [below right=of L] {$\,\Mbb$};
  \node (BH1) [gray,right=of Mb] {$B(\Hilbert_1)$};   
  \node (BH2) [gray,right=of Mbb] {$B(\Hilbert_2)$};
\draw[->] (L) edge[right hook->]            (Mbb)
              edge[left hook->]           (Mb)
        (Mb)  edge  node{$\pi$}   (Mbb);           
\draw[gray,left hook->] (Mb) -- (BH1);
\draw[gray,right hook->] (Mbb) -- (BH2);
  \end{tikzpicture}\end{center}
restricts to a \Cstar-isomorphism, hence a \Wstar-isomorphism, on the image $M$ of $\L{\infty} ([0,1],\mu).$ Thus we have projections $\{p_i\}\subset M$ summing to 1 strongly with respect to the (any) faithful nondegerate representations shown in grey in the above diagram. (See  \cite[p.898, Cor.8.iii]{DS2} for the strong convergence.) As in \cref{ex:Linf.construction.in.multipliers}, these projections are Murray-von Neumann equivalent and generate a copy of $\Oinf.$\end{example}
We summarize this as follows:
\begin{proposition} Let $\Mbk$ denote the multiplier algebra of a stable $\sigma$-unital $\Cstar$-algebra. Then
\begin{enumerate}\item There exists a  copy of $\Oinf$ in $\Mbk$ whose diagonal projections converge to 1 strictly,
\item it separates points in the multiplier algebra, and
\item it separates points in the corona algebra.
\end{enumerate}\label{cor:coronaOinf}
\end{proposition}
\begin{proof} Proof of $(1):$ The strict convergence of the diagonal projections was shown in \cref{ex:Linf.construction.in.multipliers} and the comments immediately following.

Proof of $(2)$ and $(3):$ The separation of points in the multipliers come from the strong convergence noted above and \cref{prop:strictlypositive}. Separation of points in the corona follows from \cref{prop:quotients}.\end{proof}
The above example and proposition is inspired by Kadison's work in \cite[5.1.6]{kadison1984}, and  a  similar example and application appeared independently in \cite{ShelahStephrans}. It is interesting to note that under the quotient map, the image of the masa  is still a masa in an appropriate sense (see \cite{JohnsonParrott}.)
The generalized Voiculescu theorem \cite[sect.16]{EK1} will show that  unitally embedded copies of $\Oinf$ in a corona algebra are unitarily equivalent, up to a purely $K$-theoretical obstruction, see section \ref{sect:extensions}.  The above solution will turn out to correspond to the  trivial element in the relevant group, $KK_w^1 (\Oinf,B),$ so in fact the above solution is canonical in quite a strong sense. On the other hand, these remarks will not apply in the nonstable case, and we do not know what happens in that case.
\section{Applications}\label{sect:extensions}
\subsection{Infinite sums of extensions}
 Let $A$ and $B$ be \Cstar-algebras, and let
$$
0\ \to \ B \ \to \ C \ \to \ A \ \to \ 0
$$
be an extension of $B$ by $A$ ({\it i.e.} a short exact sequence of
\Cstar-algebras).

 Recall that an extension of $B$ by $A$ is
determined by its Busby map --- the naturally associated map from $A$
to the quotient multiplier algebra, or corona algebra, of $B$, namely $\Mbb$.
The \Cstar-algebra of the extension is recovered by forming the pullback of the Busby map
and the canonical quotient map $\Mb  \to  \Mbb.$

Recall (see {\it e.g.} \cite{ElliottHandelman1989}) that, if $B$ is stable, so that the Cuntz
algebra $\O2$ may be embedded unitally in $\Mb,$ then the
Brown-Douglas-Fillmore addition of extensions, defined by
$$
\tau_1  \oplus \tau_2\ := \ s_1 \tau_1 s_1^* + s_2 \tau_2 s_2^*,
$$
where $\tau_1$ and $\tau_2$ are (the Busby maps of) two extensions of
$B$ by $A$, and $s_1$ and $s_2$ are (the images in $\Mbb$ of)
the canonical generators of $\O2$ (which are isometries with range
projections summing to 1), is compatible with Brown-Douglas-Fillmore
equivalence, defined as unitary equivalence with respect to the
unitary group of $\Mb $ --- or, rather, the image of this group in
$\Mbb$ --- and the resulting binary operation on equivalence
classes is independent of the embedding of $\O2$.
Indeed, it is sufficient to have a unital copy of $\O2$ in the corona, and this observation is used in Kirchberg's work \cite{KirchbergUnpublished}, see also  \cite{ElliottHandelman1989}.
With respect to this operation, the equivalence classes of extensions
of the \Cstar-algebra $B$ by the \Cstar-algebra $A$ form an
abelian semigroup.

 Extensions can be regarded as the basic building blocks of Kasparov's \KK-theory. In this theory as originally defined by Kasparov \cite{Kasparov1981}, only finite sums of extensions are defined. Actually,  Kasparov defined a bivariant group $\KK^1 (A,B)$ using Fredholm triples, and then pointed out an isomorphism with a group of extensions $\Ext(A,B),$ that group having the same notation but a coarser equivalence relation than in the related BDF theory \cite{BDF}. The coarser (and rather complicated) equivalence relation was needed for obtaining a useful structure called the Kasparov product. Kasparov already remarked that   a simpler equivalence relation like that of BDF could be used if the extensions were known to be absorbing. At the time it was not clear which extensions were absorbing. Baaj and Skandalis studied the existence of 6-term exact sequences in \KK, and pointed out the importance of a semisplitting map \cite{BS1989}. Skandalis defined a group called $\KKnuc^1,$ which consists of semisplit extensions with a weak nuclearity condition on the semisplitting map \cite{skandalis1998}. He pointed out that if $A$ or $B$ was nuclear, then $\KKnuc^1$ coincided with $\KK^1.$ Arveson \cite{Arveson1977} pointed out the  good stability properties of the class of semisplit extensions. In \cite{EK1} it was shown how to simplify the equivalence relation on most of the above groups, making it more like the relation used in \cite{BDF}, by adding a condition called the purely large property to the extensions considered. The purely large property was in fact a characterization of the previously mentioned property of being absorbing. This gave a picture of $\KKnuc^1(A,B)$ as purely large semisplit extensions modulo a certain form of unitary equivalence, together with a weak nuclearity condition. Further simplifications are possible with more conditions, thus for example the semisplit property is automatic if either $A$ is separable and has the LLP, and $B$ has the WEP; or  if $A$ is exact and $B$ is nuclear. See \cite{pisier2020} for more information on the LLP and the WEP.

 Both Arveson \cite{Arveson1977} and Lin \cite{Lin2002}
have made use of extensions that could be interpreted as infinite sums. Thus, Arveson considered an extension denoted $\pi\tensor\Id$ that was in effect an infinite sum. Lin gave a definition of ``diagonal extensions'' in \cite[par.\,1.4]{Lin2002}, and these were again of the form $m\tensor \Id\in \Mbk.$ Salinas \cite{Salinas1992} considered sequences of unital extensions and  a  weak topology, but he did not consider infinite sums of extensions.
Kirchberg used an infinite sum construction in a very specific way in his work on classification \cite{KirchbergUnpublished}, and we now use our results to adapt his construction slightly. In fact, it is possible that our construction is equivalent to his. A  similar construction was already considered in \cite{AAP}, both for \textsl{C*}-algebras and for von Neumann algebras.

\begin{definition}Let $\tau\colon A\rightarrow \Mbb$ be an extension. Define
$\tau_\inf (a):=\sum \vi\tau (a)\vi\star$ where the sum is taken in the corona algebra, and the $\vi$ are as in \cref{cor:coronaOinf}.\label{def:infinitesum}\end{definition}
The above sum is defined by  Theorem \ref{th:inf.sum.for.T} applied to the elements $\vi\tau(a)\vi\star,$ followed by using the partial isometry equation $\vi\vi\star\vi=\vi$ to simplify.
In the above sum all the extensions summed are the same, but it would be possible to allow them to vary. We will mostly be interested in the above sum, but since
 \cref{cor:coronaOinf} also showed that the $v_i$ come from the multiplier algebra and separate points there,   we do have an analogous definition at the level of the multiplier algebra:
\begin{definition}Let $\hat\tau\colon A\rightarrow \Mb$ be the splitting map of a trivial extension. Define
$\hat\tau_\inf (a):=\sum \vi\hat\tau (a)\vi\star$ where the sum is taken in the multiplier algebra, and the $\vi$ are as in \cref{cor:coronaOinf}.\label{def:multiplier.infinitesum}\end{definition}
To lighten the notation, we use the same symbol for $v_i$ in the multiplier algebra and in the corona algebra.

\begin{proposition} If $\tau\colon A\rightarrow \Mbb$ is a semisplit unital extension, then so is the corona sum
$\tau_\inf (a)=\sum \vi\tau (a)\vi\star$ where the $\vi$ are as in  \cref{cor:coronaOinf}. \label{prop:IsSemisplit}\end{proposition}
\begin{proof}
First of all, we must show that $\tau_\inf$ takes values in the corona, as implied above. But this follows from Theorem \ref{th:inf.sum.for.T}, with the elements $c_i$ equal to $\vi \tau(a)\vi\star.$ Since $\tau$ is a *-homomorphism, so is $\tau_\inf.$ If $\widehat\tau\colon A\rightarrow \Mb$ is a semi-splitting map for $\tau,$ then by a Stinespring theorem this semisplitting map can be taken to have the form
\begin{equation*} \hat\tau(a)=w\star \hat k(a) w \label{eq1}\end{equation*}
 for a suitable homomorphism $\hat k\colon A\rightarrow \Mbk$ and a multiplier algebra isometry $w.$ The map $\hat k$ can be taken to be the splitting map of any trivial unital absorbing extension.  (For these well-known facts, see \cite[p.353]{Arveson1977} and \cite{Kasparov1981}.)
 Applying the canonical quotient map $\pi$ into the corona, we have
 \begin{equation*} \tau(a)=w\star \pi(\hat k(a)) w \label{eq2}.\end{equation*}
 Evidently, by orthogonality of the \vi,  we have
 \begin{equation}\vi \tau(a) \vi\star = (\vi w\star )\pi(\vi\star \hat k_\infty (a)\vi) (w \vi\star)=(\vi w\star \vi\star)\pi(\hat k_\infty (a)) (\vi w \vi\star).\label{eq3}\end{equation}
Let us define $\delta_\inf (w)$ to be the infinite sum construction of Definition \ref{def:infinitesum} applied to the image in the corona of the multiplier isometry $w.$ Using  the above equation \eqref{eq3} to simplify $\delta_\inf (w\star) \pi(\hat k_\infty (a)) \delta_\inf (w)$ we conclude that
 \begin{equation*} \tau_\inf (a)= \delta_\inf (w\star) \pi(\hat k_\infty (a)) \delta_\inf (w).  \end{equation*}
  Now we lift the corona algebra element $\delta_\inf (w)$ to a contraction $\bar w$ in the multiplier algebra, and observe that $\bar w\star \hat k_\infty (a) \bar w$ gives a (completely positive) semisplitting for $\tau_\inf,$ as was to be shown.
We remark that $\bar w$ is approximately diagonal, in the sense that $p_i \bar w p_j \in B$ for $i\not=j.$ In fact, we could adjust the lifting  $\bar w$ to be diagonal by making the off-diagonal elements  zero, \emph{c.f.} Theorem \ref{th:inf.sum.for.T}.
\end{proof}

\subsection{Weak equivalence and  \Oinf}
Since, after all, the range of a extension is a subalgebra of the corona, one could wonder if unitary equivalence by unitaries from the multiplier algebra is really the most appropriate notion of equivalence to use. What if, for example, we define absorption with respect to unitary equivalence by unitaries from the corona? Calling this form of equivalence \emph{weak equivalence}, we would then consider what happens if an extension absorbs weakly nuclear trivial extensions with respect to weak equivalence. It is remarkable that this form of absorption is again characterized by the exact same algebraic property as for the more usual sort of absorption (namely, the purely large condition). We now  recall this point, and at the same time collect some useful facts about extensions.

In the forthcoming proposition, the term essential means that the Busby map is injective, the term full means that no nonzero element of the range of (some) semisplitting map is contained in a proper ideal, and the term absorbing in the nuclear sense means that the extension is unitarily equivalent to its own sum with any weakly nuclear trivial extension. We focus here on the case of unital extensions since it is needed in our application; for the nonunital case, see the remarks in \cite{gabe2016}. \begin{definition} An extension is \emph{purely large} if for every positive element $c$ of the extension algebra that is not contained in the canonical ideal $B,$ there exists a stable subalgebra $D \subset \overline {cBc}$ which is full in $B.$
\end{definition}
\begin{proposition}\label{lem:characterize.purely.large}
Let $A$ and $B$ be separable \Cstar-algebras, with
$B$ stable.
Consider  a unital essential  extension $\tau$ of $B$ by~$A$.
The following are equivalent:
\begin{enumerate}
\item[(i)] the extension  is unitally absorbing in the nuclear sense (i.e. absorbs trivial unital weakly nuclear extensions);
\item[(ii)] the extension algebra is purely large;
\item[(iii)] the image $\tau(A)$ in the corona $\Mbb$ has the property that its positive elements are  properly infinite and full; and
    \item[(iv)] the extension absorbs weakly nuclear unital trivial extensions with respect to weak equivalence (by corona unitaries).
\end{enumerate}\label{prop:omnibus.on.extensions}\end{proposition}
\begin{proof}
From \cite[Theorem, section 17]{EK1}, see also \cite[Remark 2.9]{GR2020}, we have the equivalence of (i) and~(ii).
  It is clear that (i) implies (iv). We now show that (iv) implies (iii). Thus, we are given that $\tau\oplus\sigma \sim_{w} \tau$ where $\sim_{w}$ denotes unitary equivalence by a corona unitary, and $\sigma$ is any trivial weakly nuclear extension. Setting $\sigma$ equal to Kasparov's canonical example
  of a unital, weakly nuclear, and trivial extension, we notice that the positive elements of the range of $\sigma$ are properly infinite and full. But then the positive element $\tau(a)\oplus\sigma(a)$ is properly infinite and full because it majorizes $\sigma(a)$.  Unitary equivalence preserves both fullness and the purely infinite property, so then it follows from the hypothesis that the positive elements of the range of the given extension $\tau$  are purely infinite and full, as we wanted to show.

 It remains only to show that (iii) implies (ii). The hypothesis gives that if an element $c$ from the
extension algebra is positive, then the image in the corona,  $\pi(c)$ is
properly infinite and full. To show that (ii) holds, we  show that the
algebra $A:=\overline {cBc}$ contains a stable full subalgebra.
The hypothesis implies that $\pi(c)\geqcomp
1$ by proposition 3.5 in \cite{RK}.  (We say that $p\geqcomp q$ if there exists a sequence of elements
$(r_n)$  such that $r_npr_n^{*}$ goes to $q$ in norm. See \cite{RK}.)
There is thus a sequence $r_n$ such that $r_n\pi(c)r_n^{*}\longrightarrow 1$ in norm, and therefore, choosing a
sufficiently large $n,$ the operator $r_n\pi(c)r_n^{*}$ is invertible.
Lifting to the multipliers, there is thus an $\tilde {r}$ such that
$\tilde {r}c\tilde {r}^{*}=1_{{\mathcal M}(B)}+b$. The element $b$ belongs
to the canonical ideal $B,$ which is stable, so there is a sequence of
isometries, $v_i,$ coming from stability, such that $v_i^{*}bv_i$ goes to zero in
norm\cite{HjelmborgRordam1998}. We conclude that for some index $i$, the
expression $v_i^{*}\tilde {r}c\tilde {r}^{*}v_i=1+v_i^{*}bv_i$ is close enough to 1 to be invertible,
and thus there is an $r'\in {\mathcal M}(B)$ such that $r'cr'{}^{*}=1_{
{\mathcal M}(B)}$.

Since $r'cr'{}^{*}=1_{{\mathcal M}(B)},$ the element  $V:=c^{
1/2}r^{\prime *}$ is an infinite
isometry. Since $VV^{*}\leq c\|r'\|^2$ ,  the stable full hereditary subalgebra
$\overline{VV^{*}BVV^{*}}$  is contained in $A=\overline {cBc}.$ 
This establishes the purely large property.
\end{proof}
As Skandalis already pointed out, in the above weak nuclearity is automatic if $A$ or $B$ is nuclear  \cite{skandalis1998}.
The above result shows that the purely large property is insensitive to small changes in the equivalence relation used. Thus $\KK(A,B)$ and $\KKw(A,B)$ are, to be sure, distinct groups, but they both have the same basic description as a group given by purely large extensions up to (a form of) unitary equivalence. (There exists also a related group $\KL(A,B),$ and similar comments apply.)
In certain cases, such as the group $\KK_w^1(\Oinf,B)$ mentioned earlier, the purely large condition turns out to be automatic. Both $\KK_w^1(\Oinf,B)$ and $\KK^1(\Oinf,B)$ consist of unitary equivalence classes of copies of $\Oinf,$ and differ only in the details of the unitaries which are considered.

We would like to apply these remarks to find if groups such as  $\KK^1(\Oinf,B),$ or its quotient, $\KK_w^1(\Oinf,B),$  can be computed. The UCT theorem, which is a kind of proto-index theorem,  can sometimes be used to compute $\KK^1(\Oinf,B),$ under the condition that both \Cstar-algebras are separable and the first algebra is also bootstrap class. In our case the first algebra, $\Oinf,$ is well-known to be separable, nuclear,  and bootstrap class. So the only remaining condition is that $B$ should be separable. Applying, then, the UCT theorem\cite{SchochetRosenberg},  we have the exact sequence:
\[0 \rightarrow{\Ext}(K_* (\Oinf) , K_{*}(B)) \rightarrow KK^*  (\Oinf,B) \rightarrow{\Hom}(K_*(\Oinf),K_* (B))\rightarrow 0.\]
But  $\Ext(\Z,H)=\{0\}$ and $\Hom(\Z,H)=H,$ for any abelian group $H,$ so  the above short exact sequence simplifies to
$$\KK^1  (\Oinf,B) =K_1 (B)$$ provided $B$ is separable.

Thus, we have computed $\KK^1  (\Oinf,B)$ and found that it equals $K_1(B).$ The group $\KK_w^1(\Oinf,B)$ is by definition a quotient of $\KK^1  (\Oinf,B),$ and the conclusion is that  $\KK_w^1(\Oinf,B)$ is a quotient of the usual $K$-theory group $K_1(B);$ or to phrase things differently, it is the $K_1$-group of $B$ with a
possibly  weaker than usual equivalence relation.
It is interesting and subtle that these sets of unitary equivalence classes of unitally embedded corona copies of $\Oinf$ are themselves groups, and this was the key fact that was used in the computation. We record the result as a Corollary:
\begin{corollary} Let $B$ be a separable and unital \Cstar-algebra. The unitary equivalence classes of unitally embedded copies of $\Oinf,$ in the corona of $B,$ are enumerated by the group $\KK_w^1(\Oinf,B),$  which is isomorphic to a quotient of the group $K_1(B).$ \end{corollary}
\subsection{Triviality of extensions}
We now turn to the main result of this section, which is the surprising result that even if we sum extensions which are just semisplit, the sum defines a trivial extension. Thus, in this theory, infinite sums are generally equal to zero. The key idea of the proof is that, by a Hilbert hotel argument, we have $\tau_\infty=\tau_\infty + \tau_\infty,$ and in an abelian group, an element satisfying such an equation must be the trivial element.

Now we give a generalization of  \cite[Th.1.12.i]{Lin2002}. We will not actually use the full strength of the next result, but it seems interesting to provide a criterion for a extension to be absorbing. We start by recalling a technical lemma:
\begin{lemma}[\!\protect{\cite[Th.\,2.1(e)]{HjelmborgRordam1998}}]\label{lem:stability} Let $H$ be a $\sigma$-unital \Cstar-algebra.
If there exists a countable family of orthogonal equivalent projections in $\Mult(H)$ that sums strictly to 1, then $H$ is a stable \Cstar-algebra.\end{lemma}
Now we prove the promised criterion:
\begin{lemma}[\bf An absorption criterion] Let $B$ a separable simple stable \Cstar-algebra, and let $A$ be a separable \Cstar-algebra. Let $\hat\tau\colon A\rightarrow \Mb$ be the semisplitting map of a not necessarily trivial extension, and suppose that there are countably many orthogonal and equivalent projections $P_i \in \Mb$  with respect to which $\hat\tau(a)$ is diagonal.    Then the extension is absorbing in the nuclear sense.
 \label{lem:generalized.Lin}\end{lemma}
 \begin{proof} We verify the purely large condition of  \cref{prop:omnibus.on.extensions} :  we show that the hereditary subalgebra $\overline{\hat\tau(a)B\hat\tau(a)}$ contains a nonzero stable  subalgebra that is not contained in any proper ideal of $B,$ for positive nonzero elements $a\in A.$ The condition of not being contained in a proper ideal is automatic because $B$ is assumed simple. 
 Let $H$ denote the hereditary subalgebra of the multiplier algebra $\Mb$ that is generated by the projections $\{P_i\}.$
 Consider the nonempty hereditary subalgebra $H_0 := H \bigcap \overline{\left( \hat\tau(a)B\hat\tau(a)\right) }.$ By construction, the projections $P_i$ multiply $H_0$ into $H_0$. Thus, the projections $P_i$ are in $\Mult(H_0)$ and   sum strictly to 1 there. By Lemma \ref{lem:stability} this shows that $H_0$ is stable, as we wanted.
  \end{proof}
\begin{lemma} Let $A$ be a separable unital \Cstar-algebra and let $B$ be a separable stable \Cstar-algebra. Let $\tau\colon A \rightarrow \Mbb$ be a full semisplit unital extension. Then the corona sum $\tau_\infty =\sum \vi\tau\vi\star$ is a full semisplit  extension with properly infinite range. If we assume that either $A$ or $B$ is a simple \Cstar-algebra then we can drop the assumption that the extension $\tau$ is full.
\label{lem:final.fullness.result}\end{lemma}
\begin{proof} The semisplit property was shown in Proposition \ref{prop:IsSemisplit}. A positive element $x$ is said to be properly infinite if $x$ is equivalent under generalized Murray von Neumann equivalence to $x\oplus x.$ The definition of the sum is basically the same as the BDF sum we mentioned earlier.
In the case of a positive element of the form $\tau_\infty (a),$ as in Definition \ref{def:infinitesum}, it is thus  evident, by a Hilbert hotel argument, that the properly infinite property holds.
It is clear that the sum is full if the summands are. Now let us instead assume that the \Cstar-algebra $B$ is simple. Recall that in the proof of  Proposition \ref{prop:IsSemisplit} we found an explicit semisplitting for the given extension $\tau_\infty,$ and we remarked that the semisplitting could be taken to be diagonal. But then  Lemma \ref{lem:generalized.Lin} applies to show that the given extension $\tau_\infty$ is absorbing, and this implies fullness. Finally, if the $\Cstar$-algebra $A$ is assumed simple, then the range algebra is a simple unital subalgebra of the corona, and thus every nonzero positive element of it is in the same ideal as the unit: in other words, simple range implies fullness if the range is unital.
\end{proof}

Now we obtain our triviality result.
\begin{theorem}[\bf Triviality] Let $A$ and $B$ be separable \Cstar-algebras, with $B$ nuclear and stable. If $\tau\colon A\rightarrow \Mbb$ is a semisplit unital extension, and if either $\tau$ is  full or $B$ is simple, then $\tau_\infty$ is a trivial absorbing extension.
\label{th:triviality}\end{theorem}
\begin{proof} By \cref{prop:IsSemisplit,lem:final.fullness.result}, the sum $\tau_\inf$ defines a semisplit full extension with properly infinite range. By \cref{lem:characterize.purely.large}  the extension $\tau_\inf$  is then absorbing in the nuclear sense. This means that it is an element of $KK^1(A,B),$ defined as a group of absorbing extensions under unitary equivalence by multiplier unitaries (please see the beginning of this section for a discussion). But from the construction of $\tau_\infty$ as an infinite sum, the BDF sum $\tau_\infty + \tau_\infty$ is  equivalent to $\tau_\infty.$ Thus, at the level of abelian groups, the element defined by $\tau_\infty$ is trivial. Being trivial in the group means that it is in the same unitary equivalence class as, for example, Kasparov's trivial extension. But then $\tau_\infty$ is unitarily equivalent --- by a multiplier unitary --- to a trivial (\textit{i.e.} split) extension. Thus the extension $\tau_\infty$ is a trivial (\textit{i.e.} split) extension, as was to be shown.
\end{proof}
Because of the fairly simple and general nature of the proof, it seems  probable that any reasonable way to define infinite sums of extensions will have the same property as in the above Theorem, and thus will be equivalent in $\KKnuc^1(A,B)$ to the above sum.

\subsection{Subalgebras of the corona}
We conclude by mentioning that sometimes statements about extensions become simpler if we focus on the range of the extension (in the corona) instead of  the Busby map.

\begin{corollary}[\bf locally Calkin]\label{cor:subalgebras2}
Let $B$ be a separable, nuclear, and simple \Cstar-algebra. Denote by $\corona$ the corona algebra of $B\tensor \compacts.$ A separable  subalgebra, $S,$ self-adjoint or not, of $\delta_\inf(\Mbb)$ can be rotated into the Calkin algebra by a unitary; \emph{i.e.}  there exists a unitary $U\in\Mbk$ such that $U^*  S U$ is contained
in a copy of the Calkin algebra within $\Mbkbk$.
\end{corollary}\begin{proof} We can as well replace $S$ by its \Cstar-algebra envelope.  Lifting $S$ to some separable unital subalgebra of the multiplier algebra, the quotient map into the corona  provides a semisplit extension with range  $S.$ The inclusion map of $S$ into $\Mbb$ will be denoted $\j.$ The triviality theorem (\cref{th:triviality}) provides an infinite sum $\j_\infty\colon S\to\Mbb$ which is an absorbing trivial extension.  Kasparov's trivial absorbing extension $\kappa$ takes values in $\frac{1\tensor\Bh}{1\tensor \Kh},$ a copy of the Calkin algebra within $\Mbkbk.$ As pointed out by Kasparov \cite{Kasparov1981}, see also \cite{EK1,gabe2016,GR2020}, two absorbing, unital, weakly nuclear, and trivial extensions are unitarily equivalent by a multiplier unitary, and this provides the desired unitary $U$ such that $U^* S U$ is contained in a copy of the Calkin algebra within $\Mbkbk$.
\end{proof}
It is reasonable to think of the range of the *-homomorphism $\delta_\inf(\Mbb)$   as a subalgebra of diagonal matrices, and then the above striking result says that we can move separable subalgebras of diagonal matrices into the Calkin algebra.
It seems appropriate to mention that this implies another  property of separable subalgebras of diagonal matrices. We first recall the following classic result:
\begin{lemma}[{\cite[pg. 344]{Arveson1977}}]
Consider the Calkin algebra of a separable infinite Hilbert space. If $D$ is a separable unital subalgebra of the  Calkin algebra, self-adjoint or not, and $x$ is an element of the Calkin algebra, then there exists a projection $p$ such that $p$ commutes with elements of $D$
and
\[
\dist(x,D)=\|(1-p)x p\|.
\] \label{lem:originalArveson}
\end{lemma}
Combining this lemma  with \cref{cor:subalgebras2} we have the second of the deeper properties of separable subalgebras of $\delta_\inf(\Mbb)$ :
\begin{corollary}[\bf A generalized Arveson's formula]\label{cor:arvesons.formula2}
Let $B$ be a separable, nuclear, and simple \Cstar-algebra. Denote by $\corona$ the corona algebra of $B\tensor \compacts.$  Let there be given a separable  subalgebra, $S,$ self-adjoint or not, of $\delta_\inf(\Mbb)$ and an element $x\in \delta_\inf(\Mbb).$ Then there exists a projection $p\in\corona$ such that $p$ commutes with elements of $S$
and
\[
\dist(x,S)=\|(1-p)x p\|.
\]
\end{corollary}
\begin{proof} Since being separable is equivalent to being countably generated, there exists a separable unital subalgebra $S_2$ that contains both the separable subalgebra $S$ and the element $x.$ Rotating this subalgebra $S_2$ into a Calkin algebra by \cref{cor:subalgebras2}, and applying \cref{lem:originalArveson}, we obtain a projection $p$ in the Calkin algebra. Noting that both the distance and the norm are unitarily invariant, the desired conclusion follows.\end{proof}

\section*{Thanks}
We thank Peter Evanchuck for taking a interest in this work. We thank Eberhard Kirchberg for a copy of his classic unpublished manuscript. We also thank NSERC (Canada) for financial support.

\end{document}